\documentclass[11pt]{amsart}

\usepackage{amssymb}
\usepackage{amsmath}
\usepackage{amsthm, times, float}
\usepackage{caption}
\usepackage{color}
\usepackage[]{epic,eepic}

\usepackage{dynkin-diagrams}
\usepackage{palatino}
\usepackage[margin=3cm]{geometry}

\newtheorem{theo}{Theorem}[section]

\newtheorem{prop}[theo]{Proposition}

\theoremstyle{definition}

\theoremstyle{plain}
\newtheorem{lemma}[theo]{Lemma}
\newtheorem{theorem}[theo]{Theorem}

\newtheorem{proposition}[theo]{Proposition}

\theoremstyle{definition}

\newtheorem{remark}[theo]{Remark}

\newcommand{\beq}{\begin{equation}}
\newcommand{\eeq}{\end{equation}}
\renewcommand{\a}{\alpha}
\renewcommand{\b}{\beta}

\renewcommand{\d}{\delta}
\newcommand{\e}{\epsilon}

\newcommand{\g}{\gamma}

\renewcommand{\l}{\lambda}

\renewcommand{\o}{\omega}

\newcommand{\s}{\sigma}

\newcommand{\G}{\rm{G}}
\newcommand{\M}{\rm{M}}
\newcommand{\J}{\rm{J}}

\newcommand{\gf}{\mathfrak{f}}
\renewcommand{\gg}{\mathfrak{g}}
\newcommand{\gh}{\mathfrak{h}}
\newcommand{\gk}{\mathfrak{k}}

\newcommand{\gm}{\mathfrak{m}}

\newcommand{\gp}{\mathfrak{p}}
\newcommand{\gq}{\mathfrak{q}}

\newcommand{\gt}{\mathfrak{t}}
\newcommand{\gu}{\mathfrak{u}}

\newcommand{\so}{\mathfrak{so}}
\newcommand{\su}{\mathfrak{su}}

\newcommand{\gsp}{\mathfrak{sp}}

\newcommand{\gge}{\mathfrak{e}}

\newcommand\SL{\mathrm{SL}}

\newcommand\K{\mathrm{K}}
\newcommand\U{\mathrm{U}}
\renewcommand\S{\mathrm{S}}
\newcommand\T{\mathrm{T}}

\renewcommand\sp{\mathfrak{sp}}
\renewcommand\sl{\mathfrak{sl}}

\renewcommand{\square}{\kern1pt\vbox
{\hrule height 0.6pt\hbox{\vrule width 0.6pt\hskip 3pt
\vbox{\vskip 6pt}\hskip 3pt\vrule width 0.6pt}\hrule height0.6pt}\kern1pt}

\DeclareMathOperator\End{End\;}

\DeclareMathOperator\Ad{Ad}
\DeclareMathOperator\ad{ad}

\newcommand{\n}{\nabla}

\newcommand{\be}{\begin{equation}}
\newcommand{\ee}{\end{equation}}

\def\<#1,#2>{\langle\,#1,\,#2\,\rangle}
\newcommand{\arr}{\begin{array}{rlll}}
\newcommand{\ea}{\end{array}}
\newcommand{\bea}{\begin{eqnarray}}
\newcommand{\eea}{\end{eqnarray}}
\newcommand{\bean}{\begin{eqnarray*}}
\newcommand{\eean}{\end{eqnarray*}}

\def\sideremark#1{\ifvmode\leavevmode\fi\vadjust{
\vbox to0pt{\hbox to 0pt{\hskip\hsize\hskip1em
\vbox{\hsize3cm\tiny\raggedright\pretolerance10000
\noindent #1\hfill}\hss}\vbox to8pt{\vfil}\vss}}}
\newcounter{ssig}
\newcounter{ttig}
\title[Real semisimple Lie groups and balanced metrics] {Real semisimple Lie groups and balanced metrics}

\author{Federico Giusti}
\address{Department of Mathematics, Aarhus University, Ny Munkegade 118, 8000 Aarhus C, Denmark}
\email{federico.giusti@math.au.dk}
\author{Fabio Podest\`a }
\address{Dipartimento di Matematica e Informatica "Ulisse Dini", Universit\`a di Firenze, V.le Morgagni 67/A, 50100 Firenze, Italy}
\email{fabio.podesta@unifi.it}
\date{\today}


\begin{document}

\begin{abstract} Given any non-compact real simple Lie group $\G_o$ of inner type and even dimension, we prove the existence of an invariant complex structure $\J$ and a Hermitian balanced metric with vanishing Chern scalar curvature on $\G_o$ and on any compact quotient $\M= \G_o/\Gamma$, with $\Gamma$ a cocompact lattice. We also prove that $(\M,\J)$ does not carry any pluriclosed metric, in contrast to the case of even dimensional compact Lie groups, which admit pluriclosed but not balanced metrics.   \end{abstract}

\maketitle 
\section{Introduction} 
Given a complex non-K\"ahler $n$-dimensional manifold $(\M,\J)$ it is a natural and meaningful problem to find special Hermitian metrics which might help in understanding the geometry of $\M$. Great effort has been spent in the last decades in this research topic and among special metrics the pluriclosed and the balanced conditions have shown to be highly significant.\par 
The balanced condition can be defined saying that the fundamental form $\o=h(\cdot,\J\cdot)$ of a Hermitian metric $h$ satisfies the non-linear condition $d\o^{n-1}=0$ or equivalently, $-\J \theta = \d\o=0$, where $\d$ denotes the codifferential and $\theta$ the torsion $1$-form (see e.g. \cite{Ga2}).
While this concept appears in \cite{Ga1} under the name of semi-K\"ahler (see also \cite{Gra}), in \cite{Mi} the balanced condition was started to be thoroughly investigated, highlighting also the duality with the K\"ahler condition and establishing necessary and sufficient conditions for the existence of these metrics in terms of currents. While K\"ahler metrics are obviously balanced and share with these the important relation among Laplacians $\Delta_{\partial}=\Delta_{\overline\partial}=\frac 12 \Delta$ (see \cite{Ga1}), there are many examples of non-K\"ahler manifolds carrying balanced metrics. Basic examples are given by compact complex parallelizable manifolds, which are covered by complex unimodular Lie groups $\G$ and every left invariant Hermitian metric turns out to be balanced (see \cite{AG}\cite{Ga1}\cite{Gr}). Further examples of balanced metrics are provided by any Hermitian invariant metric on a compact homogeneous flag manifold (see also \cite{FGV} for a characterization of compact homogeneous complex manifolds carrying balanced metrics) as well as by twistor spaces of certain self-dual $4$-manifolds (\cite{Mi}) and more generally (\cite{To}) by twistor spaces of compact hypercomplex manifolds (see also \cite{Fo} for other examples on toric bundles over hyperk\"ahler manifolds). Contrary to the K\"ahlerness condition, being balanced is a birational invariant (see \cite{AB1}, so that e.g. Moishezon manifolds are balanced) and compact complex manifolds $X$ which can be realized as the base of a holomorphic proper submersion $f:Y\to X$ inherit the balanced condition whenever $Y$ has it (\cite{Mi}), while the balanced property is not stable under small deformations of the complex structure (see \cite{AB2},\cite{FuY}, \cite{AU}). On the other hand, the balanced condition is obstructed, as on compact manifolds with balanced metrics no compact complex hypersurface is homologically trivial, so that for instance Calabi-Eckmann manifolds do not carry balanced metrics. This is in contrast with the fact that Gauduchon metrics, which statisfy the weaker condition  $\partial\bar\partial\o^{n-1}=0$, always exist on a compact complex manifold.  \par 
In more recent years, the rising interest in the Strominger System (see \cite{GF} and \cite{FeY} for the case of invariant solutions on complex Lie groups) has given balanced metrics a really central role in non-K\"ahler geometry, as the equivalence between the dilatino equation (i.e. one of the equations of the system) and the conformally balanced equation requires the solutions of the system to be necessarily balanced. We refer also to the work \cite{FLY}, where new examples of balanced metrics are constructed on some Calabi Yau non-K\"ahler threefolds, as well as to the results in \cite{BV}, where a new balanced flow is introduced and investigated.  \par 
The main goal of this paper is to search for invariant special Hermitian, in particular balanced, metrics in the class of semisimple real non-compact Lie groups and on their compact (non-K\"ahler) quotients by a cocompact lattice; actually it appears that, despite invariant complex structures on semisimple (reductive) Lie algebras being fully classified in \cite{Sn} (after the special case of compact Lie algebras had been considered by Samelson (\cite{Sam}) and later in \cite{Pi}), they have never been deeply investigated from this point of view. In contrast, the case of $\K$ compact is fully understood, as in such a case it is very well known that every invariant complex structure can be deformed to an invariant one for which the opposite of the Cartan-Killing form is a pluriclosed Hermitian metric $h$, i.e. it satisfies $dd^c\o_h=0$. Moreover it has been proved in \cite{FGV} that $\K$ does not carry {\it any} balanced metric at all, fueling the conjecture (\cite{FV}) that a compact complex manifold carrying two Hermitian metrics, one balanced and the other pluriclosed, must be actually K\"ahler.\par
More specifically, in this work we focus on a large class of simple non-compact real Lie algebras $\gg_o$ of even dimension, namely those which are of inner type, i.e. when the maximal compactly embedded subalgebra $\gk$ in a Cartan decomposition of $\gg_o$ contains a Cartan subalgebra. In these algebras we construct standard invariant complex structures (regular in \cite{Sn}) and write down the balanced condition for invariant Hermitian metrics. A careful analysis of the resulting equation together with some general argument on root systems allows us to show the existence of a suitable invariant complex structure and a corresponding Hermitian metric satisfying the balanced equation. By Borel's Theorem, every semisimple Lie group $\G_o$ admits a cocompact lattice $\Gamma$ so that the compact quotient $\G_o/\Gamma$ inherits the invariant balanced structure from $\G_o$. We note here that the resulting metrics come in families and moreover the same kind of arguments can be applied to show the existence of balanced structures on quotients $\G_o/\S$, where $\G_o$ is any simple non-compact Lie group of inner type of any dimension and $\S$ is a suitable abelian closed subgroup.\par 
We are also able to prove that the compact quotients $\M=\G_o/\Gamma$, endowed with the invariant complex structure that allows the existence of balanced metrics, do not carry any pluriclosed metric. This result is in accordance with the conjecture by Fino and Vezzoni and in some sense reflects a kind of duality between the compact and non-compact case, switching the existence of balanced/pluriclosed Hermitian metrics. In the last section, we prove that these balanced manifolds $\M$, despite having vanishing first Chern class, carry no non trivial holomorphic $(n,0)$-forms; furthermore  we prove that they have vanishing Chern scalar curvature. This last property may allow to better understand the geometry of these manifolds, according to some more recent results concerning the implications of vanishing Chern-scalar curvature on some geometric features (see \cite{Y}). \par
The paper is structured as follows. In Section 2, we review basic facts on simple real non-compact Lie algebras with invariant complex structures and we consider a class of invariant Hermitian metrics for which we write down the balanced condition in terms of roots. In section 3 we state our main result, namely Theorem \ref{main}, and we prove it by means of several steps. We first rewrite the balanced equation in terms of simple roots and then the key Lemma \ref{L1} allows us to select an invariant complex structure so that the relative balanced equation admits solutions. In section 4 we prove that the complex manifolds that we constructed in the previous section and that admit balanced metrics, do not carry {\it any} pluriclosed metric. In the last section, we show in Theorem 5.1 that these complex compact manifolds $(\M,\J)$ have trivial first Chern class and that the balanced metrics we have constructed have vanishing Chern scalar curvature; as a consequence we show that the Kodaira dimension $\kappa(\M)=-\infty$.

\par 
\vspace{0.5cm}
{\bf Aknowledgements.} The second author was supported by GNSAGA of INdAM and by the project PRIN 2017 ``Real and Complex Manifolds: Topology, Geometry and Holomorphic Dynamics'', n. 2017JZ2SW5.\par 
The authors would like to thank Daniele Angella for valuable conversations.  

\section{Preliminaries}

Let $\gg_o$ be a real simple $2n$-dimensional Lie algebra. It is well known that either the complexification $\gg_o^c$ is a complex simple Lie algebra (and in this case $\gg_o$ is called absolutely simple) or $\gg_o$ is the realification $\gg_\mathbb R$ of a complex simple Lie algebra $\gg$ (see e.g. \cite{He}). \par 
When $\gg_o$ is even dimensional, it is known (\cite{Mo}, see also \cite{Sas}) that $\gg_o$ admits an invariant complex structure, namely an endomorphism $\J\in \End(\gg_o)$ with $\J^2=-\rm{Id}$ and vanishing Nijenhuis tensor or, equivalently, such that 
$$\gg_o^c = \gg_o^{10}\oplus \gg_o^{01},\qquad [\gg_o^{10},\gg_o^{10}]\subseteq \gg_o^{10}.$$
If $\G_o$ is any Lie group with Lie algebra $\gg_o$, then the endomorphism $\J$ defines a (left)-invariant complex structure on $\G_o$. Moreover, thanks to a result due to Borel (\cite{Bo}), there exists a discrete, torsionfree cocompact lattice $\Gamma$ so that $\M:= \G_o/\Gamma$ is compact and the left-invariant complex structure $\J$ on 
$\G_o$ descends to a complex structure $\J$ on $\M$.\par 
We recall that when $\G_o$ is compact and even-dimensional, i.e. $\gg_o$ is of compact type, the existence of an invariant complex structure was already established by Samelson (\cite{Sam}), while in \cite{Pi} it was shown that every invariant complex structure on $\G_o$ is obtained by means of Samelson's construction. \par
If we now consider an even-dimensional $\G_o$ and a compact quotient $\M$ endowed with an invariant complex structure $\J$, we are interested in the existence of special Hermitian metrics $h$. The following proposition states a known fact, namely the non-existence of (invariant) K\"ahler structures.
\begin{prop}\label{invK} The group $\G_o$ does not admit any invariant K\"ahler metric and the compact quotient $\M=\G_o/\Gamma$ is not K\"ahler.
\end{prop} 
\begin{proof} The first assertion is contained in \cite{Ch}, but we give here an elementary proof. If $\o$ is an invariant symplectic form 
on $\gg_o$, then the closedness condition $d\o=0$ can be written as follows for $x,y,z\in\gg_o$  
\beq\label{closed}\o([x,y],z)+\o([z,x],y) + \o([y,z],x)=0.\eeq
If $B$ denotes the non-degenerate Cartan-Killing form of $\gg_o$, then we can define the endomorphism $F\in\End(\gg_o)$ by $B(Fx,y)=\o(x,y)$ ($x,y\in\gg_o$) so that $F$ turns out to be a derivation by \eqref{closed}. As $\gg_o$ is semisimple, there exists a unique $z\in \gg_o$ with $F=\ad(z)$, so that $z\in \ker\o$, a  contradiction.  \par
We now suppose that the compact manifold $\M$ has a K\"ahler metric with K\"ahler form $\o$. Using $\o$ and a symmetrization procedure that goes back to \cite{Be}, we now construct an {\it invariant} K\"ahler form on $\G_o$, obtaining a contradiction. We fix a basis $x_1,...,x_{2n}$ of $\gg_o$ and we extend each vector as a left invariant vector fields on $G_o$; these vector fields can be  projected down to $M$ as vector fields $x_1^*,\ldots,x_{2n}^*$ that span the tangent space $TM$ at each point. As $\G_o$ is semisimple, we can find a biinvariant volume form $d\mu$, that also descends to a volume form on $\M$. We now define a left-invariant non-degenerate $2$-form $\phi$ on $G_o$ by setting
$$\phi_e(x_{i},x_j) := 
\int_M \o(x_{i}^*,x_{j}^*)\ d\mu.$$
As $\mathcal L_{x_k^*}d\mu = 0$ for every $k$, we have for every $i,j,k=1,\ldots,2n$
$$\int_M x_k^*\o(x_{i}^*,x_{j}^*)\ d\mu = \int_M \mathcal L_{x_k^*}( \o(x_{i}^*,x_{j}^*)\ d\mu) = 0$$ by Stokes' theorem and therefore we obtain that 
$$d\phi(x_{i},x_j,x_k) = 
\int_M d\o(x_{i}^*,x_j^*,x_{k}^*)\ d\mu = 0.$$
This implies that $\phi$ is a symplectic form and the proof is concluded.\end{proof}
Therefore we are interested in the existence of special Hermitian metrics on the complex manifold ($\M,\J$), in particular balanced and pluriclosed metrics, when the group $\G_o$ is of non-compact type. \par 
The case of a simple Lie algebra $\gg_o$ which is the realification of a complex simple Lie algebra $\gg$ can be easily treated and will be dealt with in subsection 2.3.\par
We will now focus on some subclasses of simple real algebras, namely those which are absolutely simple and of inner type. \par
\subsection{Simple Lie algebras of inner type}\label{simple} Let $\gg_o$ be an absolutely simple real algebra of non-compact type. It is well-known that $\gg_o$ admits a Cartan decompositon 
$$\gg_o = \gk + \gp,$$
where $\gk$ is a maximal compactly embedded subalgebra and 
$$[\gk,\gp]\subseteq \gp,\quad [\gp,\gp]\subseteq \gk,$$
so that $(\gg_o,\gk)$ is a symmetric pair. Moreover the algebra $\gg_o$ is said to be {\it of inner type} when the symmetric pair $(\gg_o,\gk)$ is of inner type, i.e. when a Cartan subalgebra $\gt$ of $\gk$ is a Cartan subalgebra of $\gg_o$, i.e. its complexification $\gt^c$ is a Cartan subalgebra of $\gg_o^c$. Using the notation as in \cite{He}, p.~126, we obtain the list of all inner symmetric pairs $(\gg_o,\gk)$ of non-compact type with $\gg_o$ simple and even dimensional (Table 1).  \par
\begin{table}[ht]\label{T1}
	\centering
	\renewcommand\arraystretch{1.1}
	\begin{tabular}{|c|c|c|c|}
		\hline
		{\mbox{Type}}				& 	$\gg$					&	$\gk$	& 	{\mbox{conditions}}			 	\\ \hline \hline
	$A$ &	$\su(p,q)$ & $\su(p) + \su(q) +\mathbb R$ & $p\geq q\geq 1$, \ $p+q\ {\rm{odd}}$				\\ \hline
	$B$ & $\so(2p+1,2q)$ & $\so(2p+1) + \so(2q)$ & $p\geq 0,q\geq 1$, \ $p+q\ {\rm{even}}$			\\ \hline
	$C$ & $\gsp(2n,\mathbb R)$ & $\su(2n)+\mathbb R$ & $n\geq 1$			\\ \hline
	$C$ & $\gsp(p,q)$ & $\gsp(p) + \gsp(q)$ & $p,q\geq 1$, \ $p+q\ {\rm{even}}$			\\ \hline
	$D$ & $\so(4n)^*$ & $\su(2n)+\mathbb R$ & $n\geq 2$			\\ \hline
	$D$ & $\so(2p,2q)$ & $\so(2p) + \so(2q)$ & $p,q\geq 1,p+q\ \rm{even}\ \geq 4$			\\ \hline
	$G$ & $\gg_{2(2)}$ & $\su(2)+\su(2)$ & 			\\ \hline
	$F$ & $\gf_{4(-20)}$ & $\so(9)$ & 			\\ \hline
	$F$ & $\gf_{4(4)}$ & $\su(2)+\gsp(3)$ & 			\\ \hline
	$E$ & $\gge_{6(2)}$ & $\su(2)+\su(6)$ & 			\\ \hline
	$E$ & $\gge_{6(-14)}$ & $\so(10)+\mathbb R$ & 			\\ \hline
	$E$ & $\gge_{8(8)}$ & $\so(16)$ & 			\\ \hline
	$E$ & $\gge_{8(-24)}$ & $\su(2)+\gge_7$ & 			\\ \hline
	\end{tabular}
	\vspace{0.1cm}
	\caption{Inner symmetric pairs $(\gg,\gk)$ of non-compact type with $\gg$ simple and even dimensional.}\label{table}
\end{table}

\subsection{Invariant complex structures} In this section we will describe how to construct invariant complex structures on even-dimensional absolutely simple non-compact Lie algebras $\gg_o$. \par
We fix a maximal abelian subalgebra $\gt\subseteq \gk$, so that $\gh:= \gt^c$ is a Cartan subalgebra of $\gg:= \gg_o^c$. Note that if $\gg_o$ is even dimensional , the same holds for $\gt$. The corresponding root system is denoted by $R$ and we have the following decompositions 
$$\gk^c = \gt^c \oplus \bigoplus_{\a\in R_\gk}\gg_\a,\quad \gp^c = \bigoplus_{\a\in R_\gp}\gg_\a,$$
where a root $\a$ will be called {\it compact} (resp. {\it non-compact}), when $\gg_\a\subseteq \gk^c$ (resp. $\gg_\a\subseteq \gp^c$) and the set of all compact (resp. non-compact) roots is denoted by $R_\gk$ (resp. $R_\gp$). It is a standard fact that $\gu:= \gk + i\gp\subseteq \gg$ is a compact real form of $\gg$ and that we can choose a basis $\{E_\a\}_{\a\in R}$ of root spaces so that 
$$\tau(E_\a) = -E_{-\a}, \qquad B(E_\a,E_{-\a}) = 1,\qquad [E_\a,E_{-\a}] = H_\a$$
where $\tau$ denotes the anticomplex involution defining $\gu$, $B$ is the Cartan Killing form of $\gg$ and $H_\a$ is the $B$-dual of $\a$ (see e.g. \cite{He}). If $\s$ is the involutive anticomplex map defining $\gg_o$, we then have that 
$$\s(E_\a) = -E_{-\a},\quad \a\in R_\gk,$$
$$\s(E_\a) = E_{-\a},\quad \a\in R_\gp.$$
If we fix an ordering , namely a splitting $R = R^+\cup R^-$ with $R^-=-R^+$ and $(R^++R^+)\cap R \subseteq R^+$, we can define a subalgebra 
$$\gq := \gh_1 \oplus \bigoplus_{\a\in R^+}\gg_\a,$$
where $\gh_1\subset \gh$ is a subspace so that $\gh_1\oplus \s(\gh_1)= \gh$. The so defined subalgebra $\gq\subset \gg$ satisfies 
$$\gg = \gq \oplus \s(\gq)$$
and therefore it defines a complex structure $\J$ on $\gg_o$ with the property that $\gq = \gg_o^{10}$. This complex structure depends on the arbitrary choice of $\gh_1$, i.e. on the arbitrary choice of a complex structure on $\gt$. \par 
We remark that the complex structure $\J$ enjoys the further property of being $\ad(\gt)$-invariant, namely 
$$[\ad(x),\J] = 0,\quad x\in\gt.$$
Therefore if $\G_o$ is a Lie group with Lie algebra $\gg_o$, then $\J$ extends to a left-invariant complex structure on $\G_o$ and it will be also right-invariant with respect to right translations by elements $h\in {\rm{T}}:= \exp(\gt)$ (note that ${\rm{T}}$ might be non-compact, unless $\G_o$ has finite center). \par
We will call such an invariant complex structure {\it standard}.
\begin{remark} In \cite{Sn} the class of (simple) real Lie algebras of inner type is called ``Class I'' and it is then proved that {\it every} invariant complex structure in these algebras are standard, with respect to a suitable choice of a Cartan subalgebra (such complex structures are called regular in \cite{Sn}).
\end{remark}

\subsection{Invariant metrics and the balanced condition}\label{inv}
Let $\M$ be a compact complex manifold of the form $\G_o/\Gamma$, endowed with a complex structure $\J$ which is induced by a standard invariant complex structure $\J$ on $\G_o$, as in the previous section. It is clear that any left invariant $\J$-Hermitian metric $h$ on $\G_o$ induces an Hermitian metric $\bar h$ on $\M$ and $\bar h$ is balanced or pluriclosed if and only if $h$ is so. For the converse, we prove the following
\begin{prop} If $(\M,\J)$ admits  a balanced (pluriclosed) Hermitian metric, there exists a left invariant and right $\T$-invariant Hermitian metric on $\G_o$ which is balanced (pluriclosed resp.).  \end{prop}
\begin{proof} Suppose we have a balanced metric $h$ on $M$ with associated fundamental form $\o$. Then using the same notation and arguments as in the proof of Prop.\ref{invK}, we define a left-invariant positive $(n-1,n-1)$-form $\phi$ on $G_o$ as follows
$$\phi_e(x_{i_1},\ldots,x_{i_{2n-2}}) := 
\int_M \o^{n-1}(x_{i_1}^*,\ldots,x_{i_{2n-2}}^*)\ d\mu.$$
As $d\o^{n-1}=0$, we obtain that also $d\phi=0$. Therefore, we
can find an unique $(1,1)$-form $\hat \o$ so that $\hat\o^{n-1} = \phi$ (see \cite{Mi}) and the metric given by $\hat \o$ is balanced. As $\phi$ is left invariant, so is $\hat \omega$ by uniqueness. Now, the group $\Ad(\T)$ is compact and using a standard avaraging process we can make $\phi_e$ also $\Ad(\T)$-invariant. This means that $\phi$ is also invariant under right $\T$-translations. Again, by the uniqueness, the same will hold true for $\hat\omega$.\par As for the pluriclosed condition, the lifted metric from $\M$ to $\G_o$ is clearly pluriclosed and can be made $T$-invariant by a standard averaging. \end{proof}
\begin{remark} We can now deal with the case when $\gg_o$ is the realification of a simple Lie algebra $\gg$. In this case the complex structure $\J$ commutes with $\ad(\gg_o)$ and $\gg_o = \gu + i\gu$ is a Cartan decomposition, where $\gu$ is a compact real form of $\gg$. Let $\G_o$ be a real group with algebra $\gg_o$ and let $\U$ be the compact subgroup with algebra $\gu$. Then the metric $h$ which coincides with $-B$ on $\gu$, with $B$ on $i\gu$ and such that $h(\gu,i\gu)=0$ is a Hermitian metric which is balanced. Indeed, $h$ is $\Ad(\U)$-invariant and therefore the corresponding $\d\o$ is  $\Ad(\U)$-invariant $1$-form, hence it vanishes identically. This is consistent with the fact that complex parallelizable manifolds carry balanced metrics as they carry Chern-flat metrics, as noted in  \cite{Ga1}, p. 121 (see also \cite{AG},\cite{Gr}). \par 
On the other hand, $\G_o$ admits no invariant pluriclosed metric. Indeed, any such metric $h$ can be avaraged to produce an $\Ad(\U)$-invariant pluriclosed metric, which would be balanced by the previous argument. This is not possible, as a metric which is balanced and pluriclosed at the same time has to be K\"ahler (see e.g. \cite{AI}), contrary to Prop \ref{invK}.\end{remark}

We now focus on the case where $\gg_o$ is absolutely simple of inner type, endowed with an invariant complex structure.  
We fix a Cartan subalgebra $\gt\subseteq \gk$ with corresponding root system $R=R_\gk \cup R_\gp$ as in section \ref{simple} and we consider an ordering 
$R = R^+\cup R^-$ giving an invariant complex structure $\J_o$ on $\gg_o/\gt$. We extend $\J_o$ to an invariant complex structure $\J$ on $\gg_o$. \par 
We also fix a basis of a complement of $\gt$ in $\gg_o$
$$v_\a := \frac 1{\sqrt 2}(E_\a-E_{-\a}),\ 
w_\a := \frac i{\sqrt 2}(E_\a+E_{-\a}),\ \a\in R_\gk^+,$$
$$v_\a := \frac 1{\sqrt 2}(E_\a+E_{-\a}),\ 
w_\a := \frac i{\sqrt 2}(E_\a-E_{-\a}),\ \a\in R_\gp^+,$$
so that $v_\a,w_\a\in \gg_o$ for every $\a\in R^+$ and moreover 
$$Jv_\a = w_\a, \ Jw_\a = -v_\a,$$
$$[H,v_\a] = -i\a(H)w_\a,\ H\in\gh,$$ 
$$[v_\a,w_\a] = iH_\a,\ \a\in R_\gk^+,$$
$$[v_\a,w_\a] = -iH_\a,\ \a\in R_\gp^+.$$

We now construct invariant Hermitian metrics $h$ on $\gg_o$. First, we define $h$ on $\gt$ by choosing a $J$-Hermitian metric $h_\gt$ on $\gt$. If we set $\gm_\a := {\rm{Span}}\{v_\a,w_\a\}_{\a\in R^+}$, we define for $\a\neq \b\in R^+$
$$h(\gt,\gm_\a)= 0,\quad h(\gm_\a,\gm_\b)= 0,$$
$$h(v_\a,v_\a) = h(w_\a,w_\a) = h_\a^2,\quad h(v_\a,w_\a) = 0$$
for $h_a\in \mathbb R^+$.\par
In particular we are interested in constructing balanced Hermitian metrics, namely Hermitian metrics whose associated $(1,1)$-form $\o=h(\cdot,\J\cdot)$ satisfies $d\o^{n-1}=0$ or equivalently $\d\o=0$, where $\d$ denotes the codifferential. \par 
We use the expression 
$$\d\o (x) = -{\rm{Tr}}\nabla_{\cdot}\o(\cdot,x) = - \sum_{i}^{2n}\n_{e_i}\o(e_i,x) = $$
$$= \sum_i \o(\n_{e_i}e_i,x) + \o(e_i,\n_{e_{i}}x),$$
where $\nabla$ denotes the Levi Civita connection of $h$ and $\{e_i\}$ is an orthonormal basis of $\gg_o$ w.r.t. $h$. Note that both $h$ and $\J$ are $\ad(\gt)$-invariant and therefore $\d\o$, which is also $\ad(\gt)$-invariant, does not vanish only when evaluated on elements $x\in \gt$.\par We have the following expression for the Levi Civita connection, namely for $x,y,z\in \gg_o$ 
$$2h(\n_xy,z) = h([x,y],z) + h([z,x],y) + h([z,y],x).$$
Then for every $x\in\gt$, $y\in \gg_o$ 
$$h(\n_{y}y,x) = h([x,y],y) = 0.$$
Therefore for $x\in\gt$ we have 
\beq\label{comp}\d\o(x) = \sum_i\o(e_i,\n_{e_i}x) = -\sum_ih(Je_i,\n_{e_i}x) = \eeq
$$=-\frac 12\left( h([e_i,x],Je_i)+ 
           h([Je_i,e_i],x) + h([Je_i,x],e_i)\right).$$
We now observe that $J$ is $\ad(\gt)$-invariant and therefore 
$h([Je_i,x],e_i) = -h([e_i,x],Je_i)$ for every $i=1,\ldots,2n$, 
so that \eqref{comp} can be written as 
$$-\d\o(x) = \frac 12 \sum_ih([Je_i,e_i],x)= $$
$$= \frac 12 \cdot 2 \left(\sum_{\a\in R_\gk^+}\frac 1{h_\a^2} h([w_\a,v_\a],x)+ 
\sum_{\a\in R_\gp^+}\frac 1{h_\a^2}h([w_\a,v_\a],x)\right)=$$
$$= \sum_{\a\in R_\gk^+}\frac 1{h_\a^2}h(-iH_\a,x) + 
\sum_{\a\in R_\gp^+}\frac 1{h_\a^2}h(iH_\a,x),$$
so that 
$\d\o|_{\gt} =0$ if and only if 
$$-\sum_{\a\in R_\gk^+}\frac 1{h_\a^2}H_\a + \sum_{\a\in R_\gp^+}\frac 1{h_\a^2}H_\a=0.$$
Summing up, the metric $h$ is balanced when the following equation is satisfied
\beq\label{eq}\sum_{\a\in R_\gk^+}\frac 1{h_\a^2}\a = \sum_{\a\in R_\gp^+}\frac 1{h_\a^2}\a.\eeq\par\medskip
Note that this does {\it not} depend on the choice of the metric along the toral part $\gt$.

\section{Main result}\label{proof}
In this section we will prove our main result
\begin{theorem}\label{main} Every non-compact simple Lie group $\G_o$ of even dimension and of inner type admits an invariant complex structure $\J$ and an invariant balanced $\J$-Hermitian metric.\end{theorem}
Note that by Borel's Theorem, we can use a cocompact latice $\Gamma\subset\G_o$ to obtain compact quotients $\M=\G_o/\Gamma$, which will inherit the same balanced structure. \par

We start noting that equation \eqref{eq} involves the unknowns $\{h_\a\}_{\a\in R^+}$ and also a choice of positive roots, i.e. an ordering or equivalenty a complex structure on $\gg_o$. We will always fix a complex structure on $\gt$ once for all. It is known that giving an ordering on the root system $R$ is equivalent to the choice of a system of simple roots $\Pi$ and that two systems of simple roots are conjugate under the action of the Weyl group $W$. We may fix a system of simple roots $\Pi = \{\a_1,\ldots,\a_r\}$ and put 
$\Pi = \Pi_c \cup \Pi_{nc}$, where $\Pi_{c/nc}$ denotes the set of simple roots which are compact or noncompact. We set $\Pi_c=\{\phi_1,\ldots,\phi_k\}$, $\Pi_{nc}=\{\psi_1,\ldots,\psi_l\}$, $k+l=r = {\rm{rank}}(\gg_o)$.
Each root $\a\in R^+$ can be written as 
$$\a = \sum_{i=1}^k n_i(\a)\phi_i + \sum_{j=1}^l m_j(\a)\psi_j$$
for $n_i(\a),m_j(\a)\in\mathbb N$ nonnegative integers. If we set $g_\a:= \frac 1{h_\a^2}$ and $g_j:= g_{\phi_j}, h_j := g_{\psi_j}$, equation \eqref{eq} can be written as 
$$\sum_{\a\in R_\gk^+,\a\not\in\Pi} g_\a\left(\sum n_j(\a)\phi_j + \sum_j m_j(\a)\psi_j\right) + \sum_j g_j \phi_j = $$
$$= \sum_{\a\in R_\gp^+,\a\not\in\Pi} g_\a\left(\sum n_j(\a)\phi_j + \sum_j m_j(\a)\psi_j\right) + \sum_j h_j \psi_j,$$
and therefore 
\beq\label{sys1}
\left\{\begin{aligned}
g_j &= \sum_{\a\in R_\gp^+,\ \a\not\in\Pi} g_\a n_j(\a) &- \sum_{\a\in R_\gk^+,\ \a\not\in\Pi} g_\a n_j(a),{}\ \  &j=1,\ldots,k,\\
h_j &= \sum_{\a\in R_\gk^+,\ \a\not\in\Pi} g_\a m_j(\a) &- \sum_{\a\in R_\gp^+,\ \a\not\in\Pi} g_\a m_j(a),{}\ \  &j=1,\ldots,l.
\end{aligned}\right.\eeq
\begin{remark} If we consider for instance the case $\gg_o=\su(p,q)$ ($p+q$ even, $p,q\geq 2$) and the standard system of simple roots $\Pi=\{\e_1-\e_2,\e_2-\e_3,\ldots,\e_{p-1}-\e_p,\e_p-\e_{p+1},\ldots,\e_{p+q-1}-\e_{p+q}\}$ of $\sl(p+q,\mathbb C)$, then $\Pi_{nc}=\{\e_p-\e_{p+1}\}$ and $\Pi_c$ gives a system of simple roots for the semisimple part $\gk_{ss}$ of $\gk$. This means that every root $\a\in R_\gk^+,\a\not\in \Pi$
is a linear combination of roots in $\Pi_c$ and therefore the righthandside of the last equation in \eqref{sys1} is non-positive, so that \eqref{sys1} has no solution. This shows that the choice of the invariant complex structure might not be straightforward.   
\end{remark}

The following lemmata provide key tools in our argument.
\begin{lemma}\label{L1} For each symmetric pair $(\gg_o,\gk)$ as in Table 1, $(\gg_o,\gk)\not\cong (\so(1,2n),\so(2n))$ and given a Cartan subalgebra $\gt\subseteq \gk$ with corresponding root system $R$, there exists an ordering of the roots, hence a system of simple roots $\Pi$, such that 
\beq\label{prop} \forall \psi \in \Pi_{nc}\ \exists \psi'\in \Pi_{nc}\ {\rm{with}}\ \psi+\psi'\in R.\eeq 
\end{lemma}
This implies that, if $\Pi_{nc}=\{\psi_1,\ldots,\psi_l\}$, then for every $\psi_j\in \Pi_{nc}$ there exists $\a\in R_\gk^+$ with $m_j(\a)\neq 0$ and $\a\in {\rm{Span}}\{\Pi_{nc}\}$.\par\medskip
\noindent {\bf Remark} Note that $\sp(1,1)\cong\so(1,4)$ is also not admissible in the above Lemma. In general, for $\gg_o = \so(1,2n)$ we have the standard system $\Pi=\{\e_i-\e_{i+1},\e_n,\ i=1,\ldots,n-1\}$ with 
$\Pi_{nc}= \{\e_n\}$. As $R_c$ consists precisely of all the short roots, it is clear that for any element $\s$ of the Weyl group $W\cong \mathbb Z_2^n\ltimes \mathcal S_n$ we have that $\s(\Pi)_{nc}$ consists of one element. We will deal with this case later on.
\begin{proof} We first deal with the classical case. We start with the standard system of simple roots $\Pi$, following the notation as in \cite{He}. It is immediate to check that in this case $\Pi_{nc}$ consists of a single root $\psi$. \par 
We first deal with the case where $\psi$ is a short root. Let $\Lambda$ be the set of all simple roots which are connected to $\psi$ in the Dynkin diagram relative to $\Pi$. If $s\in W$ denotes the reflection around $\psi$, then $s$ leaves every element $\Pi\setminus \Lambda$ pointwise fixed. We observe that $\Lambda$ consists of either at most three short roots or it contains a long root.  In the first case, $s(\Lambda) = \{\psi + \lambda|\ \l\in\Lambda\}\subseteq R_\gp$ so that $s(\Pi)_{nc} = \{-\psi,s(\Lambda)\}$ and therefore the system of simple roots $s(\Pi)$ satisfies \eqref{prop}. If $\Lambda$ contains a long root, then it also contains a short root, unless $(\gg_o,\gk)=(\so(2,3),\mathbb R + \so(3))$, that is isomorphic to $(\sp(2),\gu(2))$; this case will be dealt with in the second part of the proof.  Therefore $\Lambda = \{\phi_1,\phi_2\}$ with $\phi_1$ short and $\phi_2$ long. Again the reflection $s$ around $\psi$ gives $s(\phi_1) = \psi+\phi_1$ and $s(\phi_2) = \phi_2+2\psi\in R_\gk$ or $s(\phi_2) = \psi+\phi_2\in R_\gp$. This implies that the system of simple roots $s(\Pi)$ has $s(\Pi)_{nc} = \{-\psi,\psi+\phi_1\}$ or $\{-\psi,\psi+\phi_1,\psi+\phi_2\}$ and in both cases it satisfies \eqref{prop}.\par 
We are left with the case where $\psi$ is a long root, namely the case where $\gg_o=\sp(2n,\mathbb R)$ and $\gk = \gu(2n)$. A standard system of simple roots is given by $\Pi=\{\e_1-\e_2,\e_2-\e_3,\ldots,\e_{2n-1}-\e_{2n},2\e_{2n}\}$ and $\Pi_{nc}= \{\psi=2\e_{2n}\}$. Again using $s_{\b}$, we see that $s_\b(\Pi)_{nc}=\{-2\e_{2n},\e_{2n-1}+\e_{2n}\}$ so that condition \eqref{prop} is satisfied.\par
We may now deal with the exceptional cases. Starting with the standard system of simple roots $\Pi$, we list the set $\Pi_{nc}$, that turns out to consist of a single root $\b$.  For each case, using the symmetry $s_\b$ we obtain  the system of simple roots $\Pi':=s_\b(\Pi)$ that satisfies condition \eqref{prop}.\par 
\noindent (1)\ $(\gg_o,\gk) = (\gg_2,\su(2)+\su(2))$. Here $\Pi=\{\a,\b\}$, with $\b$ long. We have $\Pi_{nc}=\{\b\}$ and $\Pi'=\{-\b,\a+\b\}$.\par
\noindent (2)\ $(\gg_o,\gk) = (\gf_{4(-20)},\so(9))$. According to \cite{He}, the standard system of simple roots is $\Pi=\{\a_1=\e_2-\e_3,\a_2=\e_3-\e_4,\a_3=\e_4,\a_4=\frac 12(e_1-\e_2-\e_3-\e_4)\}$ so that 
$\Pi_{nc}=\{\a_4\}$ and therefore  $\Pi'_{nc}=\{-\a_4,\a_4+\a_3\}$.\par 
\noindent (3)\ $(\gg_o,\gk) = (\gf_{4(4)},\su(2)+\gsp(3))$. In this case $\Pi_{nc}=\{\a_1\}$ and therefore $\Pi'_{nc}=\{-\a_1,\a_1+\a_2\}$. \par
\noindent\ (4)\ $(\gg_o,\gk)=(\gge_{8(8)},\so(16))$. For $\gge_8$ we have the standard system of simple roots 
$$\a_1=\frac 12(\e_1+\e_8)-\frac 12(\e_2+\e_3+\e_4+\e_5+\e_6+\e_7), \a_2=\e_1+\e_2, $$ 
$$\a_j=\e_{j-1}-\e_{j-2},\ j=3,\ldots,8.$$
Then $\Pi_{nc}=\{\a_1\}$ and $\Pi'_{nc}=\{-\a_1,\a_1+\a_3\}$.\par 
\noindent\ (5)\ $(\gg_o,\gk)=(\gge_{8(-24)},\su(2)+\gge_7)$. Keeping the same notation for simple roots as above, we have $\Pi_{nc} = \{\a_8\}$ and $\Pi'_{nc}=\{-\a_8,\a_8+\a_7\}$.\par
\noindent\ (6)\ $(\gg_o,\gk)=(\gge_{6(2)},\su(2)+\su(6))$. As the system root system $\Pi$ can be taken to be composed of the simple roots $\{\a_1,\ldots,\a_6\}$ of $\gge_8$, we have 
$\Pi_{nc}=\{\a_2\}$ and $\Pi'_{nc}=\{-\a_2,\a_2+\a_4\}$.\par 
\noindent\ (7)\ $(\gg_o,\gk)=(\gge_{6(-14)},\mathbb R+\so(10))$. We have $\Pi_{nc}=\{\a_1\}$ and $\Pi'_{nc}=\{-\a_1,\a_1+\a_3\}$.\end{proof}


\begin{lemma}\label{L2} For every system of simple roots $\Pi = \Pi_{c}\cup \Pi_{nc}$ 
 with $\Pi_{c}=\{\phi_1,\ldots,\phi_k\}$ we have 
$$\forall\ j=1,\ldots,k,\ \exists\ \a\in R_\gp^+,\ \a\not\in \Pi : \ n_j(\a)\neq 0,$$
 where $n_j(\a)$ denotes the coordinate of $\a$ along the root $\phi_j$.
\end{lemma}
\begin{proof} We start noting that the centralizer $C_{\gk^c}(\gp^c) = C_{\gk}(\gp)^c = \{0\}$. It then follows that $[E_{\phi_j},\gp^c]\neq \{0\}$, hence there exists $\g\in R_\gp$ with $[E_{\phi_j},E_\g]\neq 0$, i.e. $\phi_j+\g \in R_\gp$. Now, if $\g>0$, then $\a:= \phi_j+\g\in R_\gp^+\setminus\Pi$ and $n_j(\a)\geq 1$. Suppose now $\g<0$. We write $\g=c_j\phi_j + \sum_{\theta\in \Pi\setminus{\phi_j}} c_\theta\theta$ for some nonpositive integers $c_j,c_\theta$. As $\g\neq -\phi_j$, there exists at least one negative coefficient $c_\theta<0$, for some $\theta\in \Pi,\theta\neq\phi_j$. Therefore the root $\g+\phi_j$ must be negative and $1+c_j\leq 0$, i.e. $\a:=-\g\in R_\gp^+\setminus\Pi$ and $n_j(\a)=-c_j\geq 1$.
\end{proof}

We now fix a system of simple roots $\Pi$ as in Lemma \ref{L1}. In order to solve the corresponding system of equations \eqref{sys1} for the positive unknowns $\{g_i,h_j,g_\a\}$, we will show how to choose the positive values $\{g_\a\}_{\a\in R^+\setminus\Pi}$ in such a way to guarantee that the constants $\{g_i,h_j\}$,
defined to satisfy \eqref{sys1}, are positive.\par 
We set 
$$\Sigma_\gk := \{\a\in R_\gk^+|\ \a\not\in \Pi, \a\in {\rm{Span}}\{\Pi_{nc}\}\},\qquad A_\gk = (R_\gk\setminus \Pi_c) \setminus \Sigma_\gk.$$
Then the system of equations \eqref{sys1} can be written as 
\beq\label{sys2}
\left\{\begin{aligned}
g_j &= \sum_{\a\in R_\gp^+,\ \a\not\in\Pi} g_\a n_j(\a) - \sum_{\a\in A_\gk} g_\a n_j(a),{}\ \  &j=1,\ldots,k,\ &(1)\\
h_j &= \sum_{\a\in R_\gk^+,\ \a\not\in\Pi} g_\a m_j(\a) - \sum_{\a\in R_\gp^+,\ \a\not\in\Pi} g_\a m_j(a),{}\ \  &j=1,\ldots,l.\ &(2)
\end{aligned}\right.\eeq
We start assigning $g_\a=1$ for every $\a\in A_\gk$. \par 
Then, for every $j=1,\ldots,k$, we use Lemma \ref{L2} selecting a root $\a\in R_\gp^+$ with $n_j(\a)\neq 0$, $\a\not\in \Pi$. This root $\a$, which depends on $j$, contributes to the first sum in the righthandside of equation (1) in \eqref{sys2} and the value $g_\a$ can be chosen big enough so that $g_j$ is strictly positive. Summing up, we can assign values 
$\{g_\a\}_{\a\in R_\gp^+\setminus \Pi_{nc}}$ so that all $g_j$, $j=1,\ldots,k$ can be defined as in \eqref{sys2}, (1), and are strictly positive. \par 
We now turn to equation \eqref{sys2}-(2), which can now be written as 
\beq\label{2}h_j= \sum_{\a\in \Sigma_\gk} g_\a m_j(\a) +\sum_{\a\in A_\gk} m_j(\a)  - \sum_{\a\in R_\gp^+,\ \a\not\in\Pi} g_\a m_j(a),\eeq
where in the righthandside the last two sums have a fixed value. Now, by Lemma \ref{L1}, we know that for every $j=1,\ldots,l$, we can find $\a\in \Sigma_\gk$ with $m_j(\a)\neq 0$. These roots can be used to choose the coefficients $g_\a$ big enough to guarantee that $h_j>0$, when defined to satisfy \eqref{2}, is strictly positive.\par 
In order to complete the proof of our main result Theorem \eqref{main}, we are left with the case $(\gg_o,\gk)= (\so(1,2n),\so(2n))$ with standard system of simple roots $\Pi=\{\e_i-\e_{i+1},\e_n,\ i=1,\ldots,n-1\}$, $\Pi_{nc}=\{\e_n\}$. We see that 
$$R_\gk^+=\{\e_i\pm\e_j,\ i< j\},\quad R_\gp =\{\e_1,\ldots,\e_n\}.$$
Now, we use equation \eqref{eq} and search for positive real numbers $\{x,y,z_i,\ i=1,\ldots,n\}$ so that 
$$x\cdot\sum_{i<y}\e_i-\e_j + y\cdot \sum_{i< j} \e_i+e_j = \sum_{i=1}^n z_i\e_i,$$
i.e.
$$\sum_{i=1}^n [(x+y)(n-i)+(x-y)(i-1)]\e_i = \sum_{i=1}^n z_i\e_i.$$
It is clear that the above equation has positive solutions by simply choosing $x>y>0$.
\begin{remark} We can consider the metric $h_o$ which coincides with $-B$ on the compact part $\gk$, with $B$ on $\gp$ and such that $h_o(\gk,\gp)=0$. This metric is easily seen to depend only on $\gg_o$ and {\it not} on the Cartan decomposition $\gg_o=\gk+\gp$. We could then ask whether there exists a suitable complex structure such that the metric $h_o$ turns out to be balanced. The resulting equation has been already treated in \cite{AP} and has a solution if and only if $\gg_o = \su(p,p+1)\cong \su(p+1,p)$ for $p\geq 1$.   
 \end{remark}
\section{Non-existence of pluriclosed metrics}
In this section we prove the following non-existence result
\begin{proposition} The compact complex manifolds ($\M,\J$), wheer $\M=\G_o/\Gamma$, do not admit any pluriclosed metric.
 \end{proposition}
Note that in the above statement $\J$ is the complex structure we have exhibited in section \ref{proof}. \par 
Now, if $h$ is any such metric, we can obtain a pluriclosed invariant metric $h$ on $\G_o$ which is also invariant under right $\T$-translations. It follows that on $\gg$ we have 
$$h(\gg_\a,\gg_\b) = 0 \quad {\rm{if}}\quad \b\neq-\a.$$
In order to write down the condition $dd^c\o=0$, where $\o$ is the fundamental form of $h$, we recall the Koszul's formula for the differential of invariant forms. If $\phi$ is any invariant $k$-form on $\G_o$ or equivalently on $\gg_o$, then for every $v_o,\ldots,v_{k}$ in $\gg_o$ 
$$d\phi(v_o,\ldots,v_{k}) = \sum_{i<j}(-1)^{i+j}\phi([x_i,x_j],
v_1,\ldots,\widehat{v_i},\ldots,\widehat{v_j}\ldots,v_{k}).$$
We set $\phi:=d^c\o$ and compute $d\phi(E_\a,E_{-\a},E_\b,E_{-\b})$ for $\a,\b\in R^+$. We have 
$$d\phi(E_\a,E_{-\a},E_\b,E_{-\b}) = -\phi(H_\a,E_\b,E_{-\b}) +
\phi(N_{\a\b}E_{\a+\b},E_{-\a},E_{-\b})$$
$$-\phi(N_{\a,-\b}E_{\a-\b},E_{-\a},E_\b)-\phi(N_{-\a,\b}E_{\b-\a},E_\a,E_{-\b})+\phi(N_{-\a,-\b}E_{-\a-\b}E_\a,E_\b)$$
$$- \phi(H_\b,E_\a,E_{-\a}), $$
where we use the standard notation $[E_{\g},E_{\e}]=N_{\g,\e}E_{\g+\e}$ for every $\g,\e\in R$. Using the known identities for the Weyl basis (see \cite{He}, p. 172,176), we can write that 
$$d\phi(E_\a,E_{-\a},E_\b,E_{-\b}) = -\phi(H_\a,E_\b,E_{-\b}) 
- \phi(H_\b,E_\a,E_{-\a})$$
$$+2
\phi(N_{\a\b}E_{\a+\b},E_{-\a},E_{-\b}) - 2\phi(N_{\a,-\b}E_{\a-\b},E_{-\a},E_\b).$$
We also introduce the notation $JE_\g= i\e_\g E_\g$ for every $\g\in R$, where $\e_\g=\pm 1$ according to $\g\in R^{\pm}$. Then
$$dd^c\o(E_\a,E_{-\a},E_\b,E_{-\b}) =  -d\o(JH_\a,E_\b,E_{-\b}) -  d\o(JH_\b,E_\a,E_{-\a}) $$
$$- 2iN_{\a,\b} d\o(E_{\a+\b},E_{-\a},E_{-\b}) - 2iN_{\a,-\b}\e_{\a-\b}d\o(E_{\a-\b},E_{-\a},E_\b).$$ 
Now we easily compute 
$$d\o(JH_\a,E_\b,E_{-\b}) = -\o(H_\b,JH_\a)$$
and 
$$d\o(E_{\a+\b},E_{-\a},E_{-\b}) = N_{\a,\b}(\o(E_\a,E_{-\a})+\o(E_\b,E_{-\b}) - \o(E_{\a+\b},E_{-\a-\b})),$$
where we have used the fact that $N_{\a,\b}= N_{\a+\b,-\b}=-N_{\a+\b,-\a}$ (see \cite{He}, p. 172). Similarly,
$$d\o(E_{\a-\b},E_{-\a},E_\b) = N_{\a,-\b}( - \o(E_\b,E_{-\b}) + \o(E_\a,E_{-\a})-\o(E_{\a-\b},E_{\b-\a})).$$
Summing up we have 
$$dd^c\o(E_\a,E_{-\a},E_\b,E_{-\b}) = -2\o(JH_\a,H_\b)$$
$$-2iN_{\a,\b}^2(\o(E_\a,E_{-\a})+\o(E_\b,E_{-\b}) - \o(E_{\a+\b},E_{-\a-\b}))$$
$$-2iN_{\a,-\b}^2\e_{\a-\b}(- \o(E_\b,E_{-\b}) + \o(E_\a,E_{-\a})-\o(E_{\a-\b},E_{\b-\a})).$$
We now set $a_\a : h(E_\a,E_{-\a})$. The pluriclosed condition implies that 
$$0=-h(H_\a,H_\b) -iN_{\a,\b}^2(-i a_\a -i a_\b + ia_{\a+\b})$$
$$- iN_{\a,-\b}^2\e_{\a-\b}( -i\e_{\b-\a}a_{\a-\b}  + i a_\b- ia_\a)$$
hence
\beq\label{eq1} h(H_\a,H_\b) = N_{\a,\b}^2(a_{\a+\b}- a_\a - a_\b)
+ N_{\a,-\b}^2\e_{\a-\b}( \e_{\a-\b}a_{\a-\b}  +  a_\b-  a_\a)\eeq
We recall that 
$$a_\a = h(E_\a,E_{-\a}) = - h(v_\a,v_\a) < 0,\qquad \a\in R_\gk^+,$$
$$a_\a = h(E_\a,E_{-\a}) = h(v_\a,v_\a) > 0, \qquad \a\in R_\gp^+,$$
$$h(H_\a,H_\b) = -h(iH_\a,iH_\b)\in \mathbb R,\qquad 
h(H_\a,H_\a) < 0.$$

Now, we recall that the existence of the complex structure $\J$, which we constructed in section \ref{proof}, relies on Lemma \ref{L1}. In particular, when $\gg_o\neq \so(1,2n)$, we have the existence of two simple roots $\psi_1,\psi_2\in \Pi_{nc}$ with $\psi_1+\psi_2=\phi\in R_\gk$. The following lemma is elementary.
\begin{lemma} Either $\psi_1+2\psi_2\not\in R$ or $\psi_2+2\psi_1\not\in R$.\end{lemma}
\begin{proof} As $\psi_1,\psi_2$ are simple, we have $\pm(\psi_1-\psi_2)\not\in R$. Now, $\psi_i+n\psi_j\in R$ if and only if $0\leq n\leq q_j$ with 
$q_j = -2\frac{\langle \psi_1,\psi_2\rangle}{||\psi_j||^2}\in\mathbb N$ for $i\neq j$. It is then clear that $q_1,q_2\geq 2$ is impossible, as $\psi_1\neq \psi_2$ implies $q_1\cdot q_2 < 4$.\end{proof}
Suppone then that $\phi+\psi_1 = \psi_2+2\psi_1\not \in R$. We now apply \eqref{eq1} with two possible choices for $\a,\b$, namely:\par \medskip
\noindent (1)\ $\a=\psi_1,\b=\psi_2$. Then 
$$h(H_{\psi_1},H_{\psi_2}) = N_{\psi_1,\psi_2}^2(a_\phi-a_{\psi_1}-a_{\psi_2}).$$
(2)\ $\a=\phi,\b=\psi_1$. Then 
$$h(H_{\phi},H_{\psi_2}) = N_{\phi,-\psi_1}^2(a_{\psi_2}+a_{\psi_1}-a_{\phi}).$$
Subtracting (1) from (2) we get 
$$h(H_{\psi_2},H_{\psi_2}) = \left(N_{\phi,-\psi_1}^2+N_{\psi_1,\psi_2}^2\right)(a_{\psi_2}+a_{\psi_1}-a_{\phi}).$$
This is a contradiction, as $h(H_{\psi_2},H_{\psi_2})<0$, while 
$a_{\psi_i}>0$ for $i=1,2$ and $a_\phi<0$.\par 
We are left with the case $\gg_o=\so(1,2n)$, that we have dealt with separately in section \ref{proof}. In this case the complex structure $\J$ is defined by the standard system of positive roots, namely $R^+=\{\e_i\pm\e_j, \e_i,\ 1\leq i\neq j\leq n\}$. In particular $R_\gk^+ =\{\e_i\pm\e_j\}_{i\neq j}$ and $R_\gp^+=\{\e_i\}_{i=1,\ldots,n}$. We now consider $\psi_i=\e_i$, $i=1,2$, $\phi_1=\psi_1+\psi_2\in R_\gk^+$ and $\phi_2=\psi_1-\psi_2\in R_\gk^+$.   We apply \eqref{eq1} in two different ways: \par
\noindent (1)\ $\a=\psi_1,\b=\psi_2$. Then 
$$h(H_{\psi_1},H_{\psi_2}) = N_{\psi_1,\psi_2}^2(a_\phi-a_{\psi_1}-a_{\psi_2}) + N_{\psi_1,-\psi_2}^2(a_{\phi_2}+a_{\psi_2}-a_{\psi_1}).$$
(2)\ $\a=\phi_1,\b=\psi_2$. Note that $\phi_1+\psi_1\not\in R$. Then 
$$h(H_{\phi_1},H_{\psi_1}) =  N_{\phi_1,-\psi_1}^2(a_{\psi_2}+a_{\psi_1}-a_{\phi_1}).$$
Therefore 
$$h(H_{\psi_1},H_{\psi_1}) = (N_{\phi_1,-\psi_1}^2+N_{\psi_1,\psi_2}^2)(a_{\psi_2}+a_{\psi_1}-a_{\phi_1})+ N_{\psi_1,-\psi_2}^2(a_{\psi_1} -a_{\phi_2}-a_{\psi_2})$$
We now recall that, if $\g,\d\in R$, then $N_{\g,\d}^2= \frac{q(1-p)}{2}||\g||^2$, where $\d+n\g$, $p\leq n\leq q$, is the $\g$-series containing $\d$ (see \cite{He}, p.176). We then immediately see that $N_{\psi_1,\psi_2}^2 = N_{\psi_1,-\psi_2}^2$ and noting furthermore that $N_{\phi_1,-\psi_1}^2 = N_{\psi_1,\psi_2}^2$,  we can write 
$$h(H_{\psi_1},H_{\psi_1}) = N_{\psi_1,\psi_2}^2(a_{\psi_2}+3 a_{\psi_1}-2a_{\phi_1}- a_{\phi_2}),    $$
giving the contradiction $h(H_{\psi_1},H_{\psi_1}) >0$.

\section{Geometric properties}
In this section, we prove the following result, which may contribute to shed to some light on the geometry of the complex balanced manifolds we have constructed in the previous sections. 
\begin{theorem} If $(\M,\J,h)$ is a balanced $n$-dimensional manifold, where $\M=\G_o/\Gamma$, $\J$ is a standard invariant complex structure and $h$ is a balanced Hermitian metric, then the metric $h$ has vanishing Chern scalar curvature.\par 
Moreover $c_1(\M)=0$ and the Kodaira dimension $\kappa(M) = -\infty$.
 \end{theorem}

We consider a standard complex structure $\J$ on a manifold $\M = \G_o/\Gamma$. We denote by $D$ the Chern connection relative to a Hermitian metric $h$ which is induced by an invariant metric on $\G_o$, again denoted by $h$. We can moreover suppose that $h$ is invariant by the right $\T$-translations. \par 
If $x\in\gg_o$ and if we still denote by $x$ the induced left-invariant vector field on $\G_o$, we consider $D_x\in\End(\gg_o)$ the endomorphism of $\gg_o$ which assigns to every $y\in\gg_o$ the element $D_xy$ corresponding to the left invariant vector field $D_xy$. Clearly $D_x\in \so(\gg_o,h)$ and $[D_x,\J]=0$. Moreover 
\beq\label{chern}D_xy = [x,y]^{10},\qquad \forall x\in \gg_o^{01},\ y\in \gg_o^{10},\eeq
that follows from the fact that $T^{1,1}=0$, where $T$ is torsion of $D$.\par 
If $R$ denote the curvature, where $R_{xy} = [D_x,D_y]- D_{[x,y]}$, we are interested in the first Ricci tensor $\rho$ given by
$$\rho(x,y) = -\frac 12{\rm{Tr}}(J\circ R_{xy}).$$
As the complex structure and the metric are both invariant under the adjoint action of the group $T = \exp(\gt)$, we see that 
$$\rho(\gt,E_\a) = 0,\ \forall \a\in R,$$
$$\rho(E_\a,E_\b)\neq 0\ {\rm{implies}}\ \b=-\a,\ \a,\b\in R.$$
Therefore we can compute 
$$\rho(E_\a,E_{-\a}) = \frac 12 {\rm{Tr}}(JD_{H_\a}).$$
\begin{lemma} For every $x\in \gh$
$$D_x = \ad(x).$$
\end{lemma}
\begin{proof} We use similar arguments as in \cite{Po}. It will suffice to consider the case where $x\in \gh^{10}$; then for every $\a\in R^+$ we have 
$$D_xE_{-\a} = [x,E_{-\a}]^{01} = [x,E_{-\a}],\quad D_x\gh^{01} = 0$$
by \eqref{chern}. Then if $\b\in R^+$ we have 
$$h(D_xE_\a,E_{-\b}) = - h(E_\a,D_xE_{-\b}) = -\b(x)h(E_\a,E_{-\b}) = 0 \qquad {\rm{if}}\ \a\neq\b,$$
so that $D_xE_\a=\a(x)E_\a = [x,E_\a]\ ({\rm{mod}}\ \gh)$. As $h(D_xE_\a,\gh^{01}) = -h(E_\a,D_x\gh^{01}) = 0$, we conclude that 
$$D_xE_\a = [x,E_\a].$$
Finally, $h(D_x\gh^{10},\gh^{01})=0$ and $h(D_x\gh^{10},E_{-\a}) = 
-h(\gh^{10},[x,E_{-\a}]) = 0$, so that $D_x\gh = 0 = [x,\gh]$.\end{proof}

It follows that 
$$\rho(\gh,\gh)=0$$ 
and 
$$\rho(E_\a,E_{-\a}) = \frac 12\left(2\sum_{\beta\in R^+}i\beta(H_\a)\right) = B(H_\a,\delta),$$
where 
$$\delta = \sum_{\beta\in R^+} iH_\b\in \gt .$$
This means that for every $\a,\b\in R$ we have
$$\rho(E_\a,E_\b) = - B(E_\a,[\d,E_\b]) = B([E_\a,E_\b],\d)$$
and therefore for every $x,y\in \gg_o$
\beq \label{rho}\rho(x,y) = B([x,y],\delta).\eeq
This means that $\rho = d\phi$, where $\phi$ is the left-invariant $1$-form that is given by $\phi(v) = B(v,\delta)$. Then clearly $c_1(M)=0$.\par
We now show that the tensor powers $K_{\M}^{\otimes k}$ are holomorphically non trivial for every $m\geq 1$. Indeed, the metric $h$ induces a Hermitian metric on the line bundles $K_{\M}^{\otimes m}$ with curvature form $m\rho$. If $\Omega$ is a nowhere vanishing holomorphic section of $K_{\M}^{\otimes k}$ , then $m\rho =  
-i\partial\overline\partial \ln(||\Omega||^2)$. 
If we denote by\, $\widehat{}$\,  the result of the symmetrization process, which commutes with the operators $\partial$ and $\overline\partial$, we obtain on $\G_o$ that $\widehat{\rho} =-i \partial\overline\partial\widehat{\ln(||\Omega||^2)} = 0$. As $\rho$ is invariant, $\widehat\rho=\rho=0$ and we get a contradiction as $\d\ne 0$.\par 
We now compute the Chern scalar curvature $s^{Ch}$ of the metric $h$ using formula \eqref{rho}. We use the orthonormal frame $e_1,\ldots,e_{2n}$. Then 
$$s^{Ch} = \sum_i \rho(Je_i,e_i) = 
-2\sum_{\a\in R^+} \frac{1}{h_\a^2} \rho(v_\a,w_\a) = $$
$$= 2iB\left(\sum_{\a\in R_\gp^+}\frac{1}{h_\a^2} H_\a-\sum_{\a\in R_\gk^+}\frac{1}{h_\a^2} H_\a,\d\right) = 0$$
if we consider the system of positive roots satisfying equation \eqref{eq}.
The claim $\kappa(\M)=-\infty$ now follows from 
Thm. 1.4 in \cite{Y}.\par

\begin{remark} Note that also for a compact group $\K$ endowed with an invariant complex structure we have $h^{n,0}(\K)=0$, see \cite{Pi}, Prop. 3.7.\par
We finally remark here that the balanced condition implies that the two scalar curvatures that one can obtain tracing the Chern curvature tensor coincide (see \cite{Ga3}, p. 501).\end{remark}

\end{document}